\documentclass[12pt]{article}
\usepackage{amsmath}
\usepackage{amsfonts}
\usepackage{amsthm}
\usepackage{graphicx}
\usepackage{overpic}
\usepackage{a4wide}
\usepackage{amssymb} 
\usepackage{pictex}
\usepackage{rotating}  
\usepackage{color}
\usepackage{etex} 
\usepackage{enumitem}
\usepackage{accents}

\definecolor{refkey}{rgb}{0,0,1}
\definecolor{labelkey}{rgb}{1,0,0}

\usepackage{comment}

\numberwithin{equation}{section}

\newtheorem{theorem}{Theorem}[section]
\newtheorem{proposition}[theorem]{Proposition}
\newtheorem{lemma}[theorem]{Lemma}

\newtheorem{Definition}[theorem]{Definition}

\newtheorem{Remark}[theorem]{Remark}
\newenvironment{remark}{\begin{Remark}\rm}{\end{Remark}}
\newtheorem{Example}[theorem]{Example}

\newtheorem{RHproblem}[theorem]{RH problem}

\usepackage{color}

\newcommand{\C}{\mathbb{C}}

\newcommand{\R}{\mathbb{R}}

\renewcommand{\hat}{\widehat}
\renewcommand{\tilde}{\widetilde}

\def\det{\mathop{\mathrm{det}}\nolimits}

\usepackage[bookmarksopen, naturalnames]{hyperref}

\begin{document}
\title{Directional Chebyshev Constants on the Boundary}
\author{Thomas Bloom and Norman Levenberg}
\maketitle 

\begin{abstract} We prove results on existence of limits in the definition of (weighted) directional Chebyshev constants at all points of the standard simplex $\Sigma \subset \R^d$ for (locally) regular compact sets $K\subset \C^d$. 
\end{abstract}

\section{Introduction} The multivariate transfinite diameter of a compact set $K\subset \C^d$ was introduced by Leja as a generalization of the notion defined by Fekete in the univariate case ($d=1$). Let $e_{1}(z),\ldots,e_{j}(z),\ldots$ be a listing of the monomials 
$e_{j}(z)=z^{\alpha(j)}=z_{1}^{\alpha_{1}(j)}\cdots z_{d}^{\alpha_{d}(j)}$ in $\C^{d}$ where the multi-indices $\alpha(j)=(\alpha_{1}(j),...,\alpha_{d}(j))$ are arranged in some order $\prec$ with the property  that $\alpha\prec\beta$ if $|\alpha|<|\beta|$. Let $d_{n}$ be the dimension of the space $\mathcal P_{n}$ of holomorphic polynomials of degree at most $n$ in $\C^d$. For $\zeta_{1},\ldots,\zeta_{d_{n}}\in\C^{d}$, let
\begin{equation}\label{VDM}
VDM(\zeta_1,...,\zeta_{d_n}):=\det [e_i(\zeta_j)]_{i,j=1,...,d_n}  
= \det
\left[
\begin{array}{ccccc}
 e_1(\zeta_1) &e_1(\zeta_2) &\ldots  &e_1(\zeta_{d_n})\\
  \vdots  & \vdots & \ddots  & \vdots \\
e_{d_n}(\zeta_1) &e_{d_n}(\zeta_2) &\ldots  &e_{d_n}(\zeta_{d_n})
\end{array}
\right]
\end{equation}
and for a compact subset $K\subset \C^d$ let
\begin{equation}\label{Vn} 
V_n =V_n(K):=\max_{\zeta_1,...,\zeta_{d_n}\in K}|VDM(\zeta_1,...,\zeta_{d_n})|.
\end{equation}
Then
\begin{equation}\label{delta} \delta(K):= \limsup_{n\to \infty}V_{n}^{1/l_n}\end{equation}
is the transfinite diameter of $K$ where $l_n:=\sum_{j=1}^{d_n} {\rm deg}(e_j)$.

For this definition, one only requires $\alpha\prec\beta$ if $|\alpha|<|\beta|$, but for what follows, it will be convenient to use a particular order $\prec$; in this paper we will take one satisfying (by abuse of notation) $z_1\prec  z_2 \cdots \prec z_d$. Zaharjuta \cite{Z} verified the existence of the limit in (\ref{delta}) by introducing directional Chebyshev constants $\tau(K,\theta)$ and proving  
\begin{equation}\label{tdstd} \delta(K)=\exp \left(\frac{1}{|\Sigma|}\int_{\Sigma^0} \log \tau(K,\theta)d\theta \right)\end{equation}
where $\Sigma:=\{(x_1,...,x_d)\in \R^d: 0\leq x_i \leq 1, \ \sum_{j=1}^d x_i = 1\}$is the standard simplex in $\R^d$; $\Sigma^0=\{(x_1,...,x_d)\in \R^d: 0 < x_i <1, \ \sum_{j=1}^d x_i = 1\}$; and $|\Sigma|$ is the $(d-1)-$dimensional measure of $\Sigma$. If we need to specify the dimension, we will write $\Sigma^{(d)}$ for $\Sigma$. To define $\tau(K,\theta)$, for each multi-index $\alpha$ we let 
\begin{equation} \label{monicclass} M_{\alpha}:= \{p\in \C[z_1,...,z_d]: p(z)=z^{\alpha} +\sum_{\beta \prec \alpha} c_{\beta}z^{\beta}\}\end{equation}
and define the corresponding Chebyshev constants
$$\tau_{K,\alpha}:=\inf \{||p||_K:=\max_{z\in K}|p(z)|:p\in M_{\alpha}\}^{1/|\alpha|}.$$
Zaharjuta showed that for $\theta\in \Sigma^0$, the limit
$$\tau(K,\theta):= \lim_{\frac{\alpha}{|\alpha|}\to \theta}\tau_{K,\alpha}$$
exists and $\theta \to \tau(K,\theta)$ is continuous on $\Sigma^0$. 

For $\theta \in \Sigma \setminus \Sigma^0$, the limit need not exist; an elementary example was given by Zaharjuta: let $K=\{(z_1,z_2) \in \C^2: z_2=0, \ |z_1|\leq 1\}$ and $\theta =(1,0)$ (see Remark 1 in \cite{Z}). Here $K$ is pluripolar; utilizing this set one can easily construct nonpluripolar examples with the same property. However, If $K$ is sufficiently nice, our main result is that the limit exists. Given $K\subset \C^d$ compact, for $\theta\in \Sigma$ we will define
\begin{equation} \label{bdrydirl} \tau(K,\theta):= \limsup_{\frac{\alpha}{|\alpha|}\to \theta}\tau_{K,\alpha}. \end{equation}
In fact, more generally, given a continuous, admissible weight function $w$ on $K$, i.e., $w\geq 0$ with $\{z\in K: w(z)>0\}$ nonpluripolar, one can define {\it weighted} Chebyshev constants 
$$\tau^w_{K,\alpha}:=\inf \{||w^{|\alpha |}p||_K:p\in M_{\alpha}\}^{1\over |\alpha |}.$$
Then for $\theta\in \Sigma$, define 
\begin{equation} \label{wbdrydirl} \tau^w(K,\theta):= \limsup_{\frac{\alpha}{|\alpha|}\to \theta}\tau^w_{K,\alpha}. \end{equation}
It follows from \cite{[BL]} that for $\theta\in \Sigma^0$, the limit in (\ref{wbdrydirl}) exists.

\begin{theorem} \label{maint} For $K\subset \C^d$ compact and regular, if $\theta \in \Sigma \setminus \Sigma^0$, 
$$\tau(K,\theta)= \lim_{\frac{\alpha}{|\alpha|}\to \theta}\tau_{K,\alpha};$$ i.e., the limit in (\ref{bdrydirl}) exists. For $K\subset \C^d$ compact and locally regular, and $w\geq 0$ a continuous, admissible weight function on $K$, if $\theta \in \Sigma \setminus \Sigma^0$, 
$$\tau^w(K,\theta)= \lim_{\frac{\alpha}{|\alpha|}\to \theta}\tau^w_{K,\alpha};$$ i.e., the limit in (\ref{wbdrydirl}) exists.
\end{theorem}

Our interest in this result is that it provides an essential step in proving a result on convergence in probability related to random polynomials in a weighted setting \cite{BDL}.

We recall these notions of regularity for $K\subset \C^d$ compact and other relevant notions in pluripotential theory in the next section. Section 3 contains the proof of Theorem \ref{maint} and an appendix follows with proofs of two auxiliary results. 
\medskip

\noindent{\bf Acknowledgements:} We would like to thank Rafal Pierzchala for the proof of Lemma \ref{rafal}. 

\section{Pluripotential preliminaries} Recall ${\mathcal P}_n$ denotes the holomorphic polynomials of degree at most $n$ in $\C^d$. For $K \subset \C^d$ compact, we let $V_K^*(z):=\limsup_{\zeta \to z}V_K(\zeta)$ where 
$$V_K(z):=\sup\{u(z):u\in L(\C^d) \ \hbox{and} \  u\leq 0 \ \hbox{on} \ K\}$$
$$= \sup \{{1\over {\rm deg} (p)}\log |p(z)|:p\in \bigcup {\mathcal P}_n, \  ||p||_{K}\leq 1 \}.$$
Here $L(\C^d)$ is the set of all plurisubharmonic functions on $\C^d$ of logarithmic growth; i.e., $u\in L(\C^d)$ if $u$ is plurisubharmonic in $\C^d$ and $u(z)\leq \log |z| +0(1)$ as $|z|\to \infty$. It is known that $K$ is pluripolar if and only if $V_K^*\equiv +\infty$; otherwise $V_K^*\in L(\C^d)$. A compact set $K\subset \C^d$ is {\it regular} at $a\in K$ if $V_K$ is continuous at $a$; $K$ is regular if it is continuous on $K$, i.e., if $V_K=V_K^*$. A stronger notion is that of {\it local regularity}: $K$ is locally regular if it is locally regular at each of its points; i.e., for $a\in K$ and $r>0$, $V_{K\cap B(a,r)}^*(a)=0$ where $B(a,r):=\{z:|z-a|\leq r\}$. Finally, $K$ is {\it polynomially convex} if $K=\hat K$ where 
$$
  \hat K:= \{z\in \C^d:|p(z)|\leq ||p||_K,
 \ p \ {\rm polynomial}\}.
$$
Note that for any compact set $K$, $V_K=V_{\hat K}$ because $||p||_K =||p||_{\hat K}$ for any $p\in \bigcup {\mathcal P}_n$. 

Next, for a continuous, admissible weight function $w$ on $K$, let $Q:= -\log w$ and define the weighted extremal function $V^*_{K,Q}(z):=\limsup_{\zeta \to z}V_{K,Q}(\zeta)$ where
$$V_{K,Q}(z):=\sup \{u(z):u\in L(\C^d), \ u\leq Q \ \hbox{on} \ K\}$$
$$= \sup \{{1\over {\rm deg} (p)}\log |p(z)|:p\in \bigcup {\mathcal P}_n, \  ||w^{deg(p)}p||_{K}\leq 1 \}.$$
If $K$ is locally regular and $w$ is continuous, it is known that $V_{K,Q}=V_{K,Q}^*$; i.e., this function is continuous. If $K$ is only regular, this conclusion may fail. 

For a function $u\in L(\C^d)$, the {\it Robin function $\bar \rho_u$} associated to $u$ is defined as 
\begin{equation}\label{stdrobin} \bar \rho_u(z):=\limsup_{|\lambda|\to \infty} [u(\lambda z)-\log |\lambda|]. \end{equation}
Such a function is logarithmically homogeneous: $\bar \rho_u(tz)=\bar \rho_u(z)+\log |t|$ for $t\in \C$.
In particular, for $K$ nonpluripolar, we write $\bar \rho_K:=\bar \rho_{V_K^*}$ (similarly, $\bar \rho_{K,Q}:=\bar \rho_{V_{K,Q}^*}$) and define 
\begin{equation}\label{krho} K_{\rho}:=\{z\in \C^d: \bar \rho_K \leq 0\}.\end{equation}
The set $K_{\rho}$ is {\it circled}; i.e., $z\in K_{\rho}$ implies $e^{i\phi}z\in K_{\rho}$ for $\phi \in \R$. Moreover, $V_{K_{\rho}}=\max[0,\bar \rho_K]$.

We follow the notation and terminology of \cite{B2}. Given $K$ compact and a multiindex $\alpha$, we let $t_{\alpha,K}\in M_{\alpha}$ be a Chebyshev polynomial; i.e., $||t_{\alpha,K}||=\tau_{K,\alpha}^{|\alpha|}$. More generally, if $w\geq 0$ is an admissible weight on $K$, we let $t^w_{\alpha,K}\in M_{\alpha}$ be a weighted Chebyshev polynomial; i.e., $||t^w_{\alpha,K}||=(\tau^w_{K,\alpha})^{|\alpha|}$. For a polynomial $p(z)=\sum_{|\alpha|\leq d} c_{\alpha}z^{\alpha}$ with at least one $c_{\alpha}\not =0$ when $|\alpha|=d$, we let $\hat p(z)= \sum_{|\alpha| = d} c_{\alpha}z^{\alpha}$. For such $p$, letting $u(z):=\frac{1}{d} \log |p(z)|$, we have $u\in  L(\C^d)$ and $\bar \rho_u(z)= \frac{1}{d} \log |\hat p(z)|$. Finally, given a homogeneous polynomial $Q$ of degree $d$, we let $Tch_K Q$ denote a (Chebyshev) polynomial satisfying $\widehat {Tch_K Q}=Q$ and 
$$||Tch_K Q||_K=\inf\{ ||Q+h||_K: h\in \mathcal P_{d-1}\}.$$

 \section{Directional Chebyshev constants on the boundary}
The proof of our main result, Theorem \ref{maint}, consists of three steps:

\begin{enumerate}
\item the limit in (\ref{bdrydirl}) exists for $K$ circled and regular in $\C^d$ implies the limit in (\ref{bdrydirl}) exists for $K$ compact and regular in $\C^d$ (Proposition \ref{step1});

\item the limit in (\ref{bdrydirl}) exists for $K$ compact and regular in $\C^d$ implies the limit (\ref{wbdrydirl}) in the weighted case exists for $K$ compact and locally regular in $\C^d$ and $w$ admissible and continuous on $K$ (Proposition \ref{step2});

\item  the limit (\ref{wbdrydirl}) in the weighted case exists for $K$ compact and locally regular in $\C^{d-1}$ and $w$ admissible and  and continuous on $K$ implies the limit in (\ref{bdrydirl}) exists for $K$ circled and regular in $\C^d$ (Proposition \ref{step3}).

\end{enumerate}

\noindent From 1.-3. and the standard fact that, for $K\subset \C$ compact and $w$ admissible on $K$, the limit $\lim_{n\to \infty} \tau_{K,n}^w$ of weighted Chebyshev constants 
$$\tau_{K,n}^w:=\bigl(\inf \{||w^np_n||_K: p_n(z) = z^n + ...\}\bigr)^{1/n}$$
exists (cf., Theorem III.3.1 in \cite{SaTo}), Theorem \ref{maint} follows. We summarize the argument at the end of this section.
 
Let $K\subset \mathbb{C}^d$ be a compact, regular set. Our first result relates directional Chebyshev constants of $K$ with those of $K_{\rho}$ (see (\ref{krho})). We note that if $K$ is regular, it is known that $\bar \rho_K$ is continuous and $K_{\rho}$ is regular (cf., Corollary 4.4 of \cite{BLM}). Fix $\theta\in\Sigma$. Let $\{\alpha(j)\}$ for $j=1,2...$
 be a sequence of $d-$multi-indices with
\begin{equation}\label{lim}
 \lim_{j \to \infty}\frac{\alpha(j)}{|\alpha(j)|}=\theta.
 \end{equation}

 \begin{proposition} \label{step1} Let $K\subset \mathbb{C}^d$ be a compact, regular set. 
 Suppose that for all sequences satisfying (\ref{lim}) 
$$\lim_{j \to \infty}||t_{\alpha(j),K_{\rho}}||^{\frac{1}{|\alpha(j)|}}_{K_{\rho}}=:\tau(K_{\rho},\theta)$$ 
exists and all the limits are equal. Then for all sequences satisfying (\ref{lim})
 $$\lim_{j \to \infty}||t_{\alpha(j),K}||^{\frac{1}{|\alpha(j)|}}=:\tau(K,\theta)$$
 exists and all the limits are equal. Furthermore, $\tau(K_{\rho},\theta)$=$\tau(K,\theta)$. 
\end{proposition}

\begin{proof} 
From \cite{B2}, equation (3.7), we have
$$||\hat{t}_{\alpha(j),K}||_{K_{\rho}}\leq ||t_{\alpha(j),K}||_K.$$
Since $\hat{t}_{\alpha(j),K}$ is a candidate for $t_{\alpha(j),K_{\rho}}$
$$ ||t_{\alpha(j),K_{\rho}}||_{K_{\rho}}\leq
||\hat{t}_{\alpha(j),K}||_{K_{\rho}}.$$
From the previous two inequalities, together with our hypothesis, we have
\begin{equation}\label{limlower}
\tau(K_{\rho},\theta)=\lim_{j \to \infty}||t_{\alpha(j),K_{\rho}}||^{\frac{1}{|\alpha(j)|}}_{K_{\rho}}\leq \liminf_{j \to \infty}    ||t_{\alpha(j),K}||_K^{\frac{1}{|\alpha(j)|}}.  
\end{equation}
Now $K_{\rho}$ is circled; thus $t_{\alpha(j),K_{\rho}}$ may be taken to be homogeneous of degree $|\alpha(j)|$. Moreover, the extremal function for $K_{\rho}$ is given by $V_{K_{\rho}}=\max [0,\bar \rho_K]$.

Thus,
$$\frac{1}{|\alpha(j)|}\log|t_{\alpha(j),K_{\rho}}(z)|
 \leq\frac{1}{|\alpha(j)|}\log||t_{\alpha(j),K_{\rho}}||_{K_{\rho}} +\bar \rho_K(z);$$
hence
 $$\limsup_{j \to \infty}\frac{1}{|\alpha(j)|}\log|t_{\alpha(j),K_{\rho}}(z)|
 \leq\limsup_{j \to \infty}\frac{1}{|\alpha(j)|}\log||t_{\alpha(j),K_{\rho}}||_{K_{\rho}} +\bar \rho_K(z)$$
for all $z\in\mathbb{C}^d$.

Using the hypothesis that the limit on the right of the above equation exists, as in \cite{B2}, equation (3.11), we have
$$\limsup_{j \to \infty}\frac{1}{|\alpha(j)|}\log||Tch_Kt_{\alpha(j),K_{\rho}}||_K
 \leq\log\tau(K_{\rho},\theta).$$
But by the defining property of $t_{\alpha(j),K}$,
    $$||t_{\alpha(j),K}||_K\leq||Tch_Kt_{\alpha(j),K_{\rho}}||_K$$
 so that 
 \begin{equation}\label{limupper}
 \limsup ||t_{\alpha(j),K}||_K^{\frac{1}{|\alpha(j)|}}\leq \tau(K_{\rho},\theta).   \end{equation}
Combining (\ref{limupper}) and (\ref{limlower}) we conclude.
\end{proof}
\medskip

For the next result we will use notation and results from \cite{BL2}, particularly Proposition 3.4. Let $K\in\mathbb{C}^d$ be a locally regular compact set and $w=e^{-Q}$ a continuous admissible weight function on $K$. Let
$$Z=Z(K):=\{z\in\mathbb{C}^d : V_{K,Q}(z) \leq M=M(K):=||V_{K,Q}||_K\}.$$
From the proof of Theorem 2.5 in \cite{BL2}, $V_Z$ is continuous, i.e., $Z$ is regular, and we have $V_{K,Q}=V_Z+M$
on $\mathbb{C}^d\setminus Z$. By Proposition 3.4 of \cite{BL2}, for $\theta\in\Sigma^0$ we have 
$$\tau(Z,\theta)=e^{M}\tau^w(K,\theta).$$
We will show:

\begin{proposition} \label{step2} Let $K$ be compact and locally regular in $\C^d$ and let $w=e^{-Q}$ be admissible and continuous on $K$. 
Suppose that for all sequences satisfying (\ref{lim})
$$\lim_{j \to \infty}||t_{\alpha(j),Z}||_Z^{\frac{1}{|\alpha(j)|}}=:\tau(Z,\theta)$$
exists and all limits are equal. Then for all sequences satisfying (\ref{lim}) all the limits 
 $$\lim_{j \to \infty}||w^{|\alpha(j)|}t^w_{\alpha(j),K}||_K^{\frac{1}{|\alpha(j)|}}=:\tau^w(K,\theta)$$
exist and are equal. Furthermore, $\tau(Z,\theta)=e^{M}\tau^w(K,\theta)$.

\end{proposition}
\begin{proof}

 The proof follows closely the proof of Proposition 3.4 of \cite{BL2}. First, $t^w_{\alpha(j),K}\in M_{\alpha(j)}$, hence it is a candidate for $t_{\alpha(j),Z}$ so that 
 \begin{equation}\label{can}
||t_{\alpha(j),Z}||_Z^{\frac{1}{|\alpha|(j)|}}\leq ||t^w_{\alpha(j),K}||_Z^ {\frac{1}{|\alpha|(j)|}}.
\end{equation}
Let 
$$
p_j(z):=\frac{t^w_{\alpha(j),K}(z)}{||w^{|\alpha(j)|}t^w_{\alpha(j),K}||_K}.
$$
Then
$$
\frac{1}{|\alpha(j)|}\log |p_j(z)|=\frac{1}{|\alpha(j)|}
\log |t^w_{\alpha(j),K}(z)|-\frac{1}{|\alpha(j)|}\log|| w^{|\alpha(j)|}t^w_{\alpha(j),K}||_K.
$$
Since $||w^{|\alpha(j)|}p_j||_K=1$, combining this with the defining property of $Z$ we obtain
$$
 \frac{1}{|\alpha(j)|}\log ||t^w_{\alpha(j),K}||_Z\leq  M+\frac{1}{|\alpha(j)|}\log|| w^{|\alpha(j)|}t^w_{\alpha(j),K}||_K.
$$
 Taking exponentials and combining with (\ref{can}) we have 
$$
||t_{\alpha(j),Z}||_Z^{\frac{1}{|\alpha|(j)}}\leq ||w^{|\alpha(j)|}t^w_{\alpha(j),K}||_K^{\frac{1}{|\alpha(j)|}}e^M.
$$
This yields
\begin{equation}\label{liminf}
\tau(Z,\theta)\leq \liminf_{j \to \infty}||w^{|\alpha(j)|}t^w_{\alpha(j),K}||_K^{\frac{1}{|\alpha(j)|}}e^M.
 \end{equation}
The proof will be completed by showing that for all sequences satisfying (\ref{lim}) that
\begin{equation}\label{limsup2}
\tau(Z,\theta)\geq \limsup_{j \to \infty}||w^{|\alpha(j)|}t^w_{\alpha(j),K}||_K^{\frac{1}{|\alpha(j)|}}e^M.
 \end{equation}
 To this end, note that
$$
  \frac{1}{|\alpha(j)|}\log |t_{\alpha(j),Z}(z)|\leq  V_Z(z)+\frac{1}{|\alpha(j)|}\log||t_{\alpha(j),Z}||_Z.
$$
Taking the Robin function of both sides of the equation
we have
$$
\frac{1}{|\alpha(j)|}\log |\hat{t}_{\alpha(j),Z}(z)|\leq \bar \rho_Z(z)
+ \frac{1}{|\alpha(j)|}\log||t_{\alpha(j),Z}||_Z.
$$ 
Since $V_{K,Q}=V_Z+M$
on $\mathbb{C}^d\setminus Z$ we have $\bar \rho_{K,Q}=\bar \rho_Z +M$, hence 
$$
\frac{1}{|\alpha(j)|}\log |\hat{t}_{\alpha(j),Z}(z)|\leq \bar \rho_{K,Q}(z)-M
+ \frac{1}{|\alpha(j)|}\log||t_{\alpha(j),Z}||_Z.
$$
Taking limsup of both sides and recalling that the limit exists on the right side, we obtain
$$
\limsup_{j \to \infty}\frac{1}{|\alpha(j)|}\log |\hat{t}_{\alpha(j),Z}(z)|\leq \bar \rho_{K,Q}(z)-M
+ \log \tau(Z,\theta).
$$
By Theorem 3.3 of \cite{BL2},
$$
\limsup_{j \to \infty}||w^{|\alpha(j)|}Tch_{K,w}\hat{t}_{\alpha(j),Z}||_K^{\frac{1}{|\alpha_j|}}
\leq e^{-M}\tau(Z,\theta).
$$
This establishes (\ref{limsup2}) since $Tch_{K,w}\hat{t}_{\alpha(j),Z}$ is a candidate for 
$t^w_{\alpha(j),K}$.

\end{proof}

\begin{remark} This result remains valid if $K$ is only regular if we assume $V_{K,Q}$ is continuous for the proof of  Theorem 2.5 in \cite{BL2} is still valid in this setting to show $V_Z$ is continuous and $V_{K,Q}=V_Z+M$ on $\mathbb{C}^d\setminus Z$.

\end{remark}

The third and final step will require a bit more work. In the proof of the next result, we defer proofs of two technical lemmas to an appendix. 

\begin{proposition} \label{step3}
Suppose for $M\subset \mathbb{C}^{d-1}$ compact and locally regular and $w=e^{-Q}$ any continuous admissible weight function on $M$, we have that for any
$\theta' =(\theta'_2,... ,\theta'_d) \in\Sigma^{(d-1)}$, and for any sequence of $d-1$ multi-indices satisfying 
$\alpha'(j)=(\alpha'_2(j),...,\alpha'_d(j))$ with
$\lim_{j\rightarrow{\infty}}\frac{\alpha'(j)}{|\alpha'(j)|}\rightarrow{\theta'}$, 
$$\lim_{j\rightarrow{\infty}}||w^{|\alpha'(j)|}t^{w}_{\alpha'(j),M}||_M^{1/{|\alpha'(j)|}}$$
exists. Then for any regular, circled compact set $K\subset\mathbb{C}^d$ and
 for any sequence of $d$ multi-indices $\alpha(j)=(\alpha_1(j),...,\alpha_d(j))$ with  $j=1,2,...$ and
 $\frac{\alpha(j)}{|\alpha(j)|} \rightarrow\theta =(\theta_1,\theta')\in\Sigma^{(d)}$ where $\theta_1 <1$, 
 $$\lim_{j \to \infty}||t_{\alpha(j),K}||_K^\frac{1}{|\alpha(j)|}$$
exists.
\end{proposition}
\begin{proof} Since Chebyshev polynomials for $K$ and $\hat K$ are the same, we may assume $K=\hat K$. Moreover,  for circled $K$, $\hat K$ coincides with the homogeneous polynomial hull. It follows from (\ref{krho}) that $K=\{z\in \C^d: \bar \rho_K \leq 0\}$ where $\bar \rho_K $ is the logarithmic  
homogeneous Robin function of $K$. For $\epsilon >0$, we set $K_{\epsilon}:=\{z\in \C^d: \bar \rho_K(z) \leq-\epsilon\}$. Using Lemma \ref{L} in the appendix, for some $\eta =\eta(\epsilon) >0$, the polynomial convex hull 
$\widehat{S}$ of $S=S(\epsilon,\eta):=K\setminus \{|z_1|<\eta\}$ contains $K_{\epsilon}$. Note that $S$ is compact, circled and nonpluripolar. Moreover, it is known that we can approximate $\hat S$ from the outside by a sequence of locally regular homogeneous polynomial polyhedra (cf., p. 93 of \cite{BL2}). Thus we may assume $\hat S$ is locally regular and we still have the inclusions $$K_{\epsilon}\subset \widehat{S}\subset
\hat K.$$
Thus we obtain the following inequalities on Chebyshev polynomials:
\begin{equation}\label{tch}
||t_{\alpha,K_{\epsilon}}||_{K_{\epsilon}}\leq
||t_{\alpha,S}||_S\leq||t_{\alpha,K}||_K.
\end{equation}
The scaling $z\rightarrow e^{-\epsilon}z$ maps $K$ onto $K_{\epsilon}$ which yields that
\begin{equation}\label{tch2}
||t_{\alpha,K_{\epsilon}}||_{K_{\epsilon}}=e^{-\epsilon|\alpha|}||t_{\alpha,K}||_K. 
\end{equation}
Thus if, under condition (\ref{lim}) that $\lim_{j \to \infty}\frac{\alpha(j)}{|\alpha(j)|}=\theta \in \Sigma^{(d)}$, we knew that 
\begin{equation}\label{tt}
\lim_{j \to \infty}||t_{\alpha(j),S}||_S^{\frac{1}{|\alpha(j)|}}
\end{equation}
exists, then $\lim_{j \to \infty}||t_{\alpha(j),K}||_K^{\frac{1}{|\alpha(j)|}}$ also exists, since by (\ref{tch}) and (\ref{tch2})
$$\liminf_{j \to \infty}||t_{\alpha(j),K}||_K^{\frac{1}{|\alpha(j)|}}\geq \lim_{j \to \infty}||t_{\alpha(j),S}||_S^{\frac{1}{|\alpha(j)|}}\geq
\limsup_{j \to \infty}||t_{\alpha(j),K_{\epsilon}}||_{K_\epsilon}^{\frac{1}{|\alpha(j)|}}$$
$$=e^{-\epsilon|\alpha(j)|}\limsup_{j \to \infty}||t_{\alpha(j),K}||_K^{\frac{1}{|\alpha(j)|}},$$
which holds for every $\epsilon>0$. 

It remains to show that the limit in (\ref{tt}) exists. Suppose $\alpha=(\alpha_1,\alpha_2,...,\alpha_d)$ is a $d$ multiindex. Recall we use $\prec$ satisfying (by abuse of notation) $z_1\prec  z_2 \cdots \prec z_d$. Then we may write $t_{\alpha,S}$ as a sum of monomials using this lexicographic order:
$$t_{\alpha,S}(z)=z_1^{\alpha_1}z_2^{\alpha_2}\cdots z_d^{\alpha_d}+...+cz_1^{|\alpha|}.$$
Thus 
$$|t_{\alpha,S}(z)|=|z_1|^{|\alpha|} |(\frac{z_2}{z_1})^{\alpha_2}\cdots (\frac{z_d}{z_1})^{\alpha_d}+...+c|=:|z_1|^{|\alpha|} \ |r_{{\alpha '}}(t_1,...,t_{d-1})| $$ 
where $t_j=\frac{z_{j+1}}{z_1}$ for $j=1,...,d-1$ and ${\alpha}'=
(\alpha_2,...,\alpha_d)$ is a $d-1$ multiindex. Consider the map $\phi :\mathbb{C}^d\rightarrow  \mathbb{C}^{d-1}$
given by
$$(z_1,...,z_d)\rightarrow \Bigl(\frac{z_2}{z_1},...,\frac{z_d}{z_1}\Bigr)=(t_2,...,t_{d-1}).$$
We let $L=\phi(\hat S)$. Then $L$ is compact since $\hat S$ is and $|z_1|\geq\eta$ on $\hat S$. Moreover, since $\hat S$ is contained in $|z_1|\geq \eta$, $\phi$ is holomorphic with complex Jacobian of maximal rank in a neighborhood of $\hat S$; thus, by Lemma \ref{rafal} in the appendix, $L$ is locally regular. 
The lexicographic ordering in $\mathbb{C}^{d-1}$ is determined by $t_1\prec t_2 \cdots \prec t_{d-1}$. Then the polynomial 
$r_{\alpha '}\in M_{\alpha '}$ and hence it is a candidate for $t_{\alpha ',L}$. Since $z^{\alpha}$ is a candidate for $t_{\alpha,S}$ we have $||t_{\alpha,S}||_S \leq B^{|\alpha|}$ for some $B>0$. Also, since $|z_1|$ is bounded below by a positive constant on $L$,
$||r_{\alpha '}||_L\leq C^{|\alpha|}$ for a positive constant $C$.

We now consider a sequence of $d$ multi-indices $\alpha (j)=(\alpha(j)_1,...,\alpha(j)_d), \  j=1,2,...$ such that  
 $$\lim_{j\to \infty} \frac{\alpha(j)}{|\alpha(j)|}\rightarrow \theta=(\theta_1,...,\theta_d).$$       
We assume that $\theta_1<1$. Then $\alpha' (j)=(\alpha(j)_2,...,\alpha(j)_d), \ j=1,2,..$
is a sequence of $d-1$ multiindices and we set $\theta'=(\theta_2,...,\theta_d)$. Note that $\theta'\not \in \Sigma^{(d-1)}$ but $\frac{\theta'}{|\theta'|} \in \Sigma^{(d-1)}$. We will show that (\ref{tt}) exists by showing that
\begin{equation}\label{threeten}
\lim_{j \to \infty}|||z_1|^{\alpha(j)|}r_{\alpha'(j)}(t_1,...,t_{d-1})||_L^{\frac{1}{|\alpha(j)|}}
\end{equation}
exists.

To this end, note that
\begin{equation}\label{thetas} \lim_{j \to \infty} \frac{|\alpha(j)|}{|\alpha'(j)|}= \frac{\theta_1+...+\theta_d}{\theta_2+...+\theta_d}=\frac{1}{|\theta'|} \end{equation}
so that
\begin{equation}\label{uniform}
|z_1|^{\frac{|\alpha(j)|}{|\alpha'(j)|}}|z_1|^{\frac{-1}{|\theta'|} }\rightarrow 1
\end{equation}
uniformly on $L$. Let $v=v(t_1,...,t_{d-1}):=|z_1|^{\frac{1}{|\theta'|} }$ for $(t_1,...,t_{d-1})\in L$; this is a positive, continuous function and hence an admissible weight on $L$. 
Thus by the hypotheses of the proposition, we have the existence of the limit
$$\lim_{j \to \infty}||v^{\alpha'(j)}t_{\alpha'(j),L}||_L^{\frac{1}{|\alpha'(j)|}}=\tau^v(\frac{\theta'}{|\theta'|} ,L).$$
Using this together with (\ref{uniform}) we deduce that
\begin{equation}\label{22}
\lim_{j \to \infty}|||z_1|^{|\alpha(j)|}t_{\alpha'(j),L}||_L^{\frac{1}{|\alpha'(j)|}}=\lim_{j \to \infty}||(|z_1|^{\frac{|\alpha(j)|}{|\alpha'(j)|}})^{|\alpha'(j)|}t_{\alpha'(j),L}||_L^{\frac{1}{|\alpha'(j)|}}=\tau^v(\frac{\theta'}{|\theta'|} ,L)
\end{equation}

We want to show that
$$\lim_{j \to \infty}|||z_1|^{|\alpha(j)|}r_{\alpha'(j)}(t_1,...,t_{d-1})||_L^{\frac{1}{|\alpha(j)|}} $$
exists; from (\ref{thetas}) it suffices to show (recall (\ref{threeten})) that 
$$\lim_{j \to \infty}|||z_1|^{|\alpha(j)|}r_{\alpha'(j)}(t_1,...,t_{d-1})||_L^{\frac{1}{|\alpha'(j)|}} $$
exists. We proceed by contradiction. If the limit does not exist there is a sequence $\alpha(j)$
such that
$$\lim_{j \to \infty}|||z_1|^{|\alpha(j)|}r_{\alpha'(j)}(t_1,...,t_{d-1})||_L^{\frac{1}{|\alpha'(j)|}} >\tau^v(\frac{\theta'}{|\theta'|} ,L).  $$
Such a limit cannot be smaller than $\tau^v(\frac{\theta'}{|\theta'|} ,L)  $
by (\ref{22}). Thus 
 $$|||z_1|^{|\alpha(j)|}|r_{\alpha'(j)}(t_1,...,t_{d-1})|||_L>|||z_1|^{|\alpha(j)|}t_{\alpha'(j),L}||_L$$
for $|\alpha(j)|$ large in the sequence.
Composing with $\phi$ we have
$$||t_{\alpha(j),S}||_S>|||z_1|^{|\alpha(j)|}t_{\alpha'(j),L}(\phi)||_S$$
which contradicts the minimality of the Chebyshev polynomials $t_{\alpha(j),S}$ 
since the polynomial $z_1^{|\alpha(j)|}t_{\alpha'(j),L}(\phi)$ is in $M_{\alpha(j)})$.

\end{proof}

This proves Step 3 in all cases except the case $\theta=(1,0,0,...,0)$. We verify this 
directly in the next lemma.

\begin{lemma}\label{100} Let $K\subset \C^d$ be compact, circled and regular. Let $\{\alpha(j)\}$ for $j=1,2...$
 be a sequence of $d-$multi-indices with
$$ \lim_{j \to \infty}\frac{\alpha(j)}{|\alpha(j)|}=(1,0,0,...,0).$$
Then 
$$\lim_{j \to \infty} ||t_{\alpha(j),K}||_K^\frac{1}{|\alpha(j)|}$$
exists and is the same for all such $\{\alpha(j)\}$.

\end{lemma}

\begin{proof} We begin with some general facts in Zaharjuta's paper \cite{Z}. In Lemma 3, he shows for an arbitrary compact set $K$ that for $\theta \in \Sigma \setminus \Sigma^0$ one has  
$$ \tau_{-}(K,\theta):=\liminf_{\frac{\alpha}{|\alpha|}\to \theta}\tau_{K,\alpha}= \liminf_{\theta'\to \theta, \ \theta'\in \Sigma^0}\tau(K,\theta').$$
For the reader's convenience, we give his argument to show that
\begin{equation}\label{liminf} \liminf_{\frac{\alpha}{|\alpha|}\to \theta}\tau_{K,\alpha}\leq  \liminf_{\theta'\to \theta, \ \theta'\in \Sigma^0}\tau(K,\theta').\end{equation}
Indeed, choosing a sequence $\{\theta_k\}\subset \Sigma^0$ with $\theta_k \to \theta$, we can find a sequence of integers $\{i_k\}$ with $i_k\to \infty$ such that
$$|\frac{\alpha(i_k)}{|\alpha(i_k)|}-\theta_k|<1/k \ \hbox{and} \ |\tau_{K,\alpha(i_k)}-\tau(K,\theta_k)|< 1/k.$$
Then
$$\liminf_{\frac{\alpha}{|\alpha|}\to \theta}\tau_{K,\alpha}\leq \liminf_{k\to \infty} \tau_{K,\alpha(i_k)}=\liminf_{k\to \infty}\tau(K,\theta_k).$$
Since this holds for any sequence $\{\theta_k\}\subset \Sigma^0$ with $\theta_k \to \theta$, (\ref{liminf}) follows. 

A similar proof shows that for $\theta \in \Sigma \setminus \Sigma^0$
\begin{equation}\label{limsup} \tau(K,\theta):= \limsup_{\frac{\alpha}{|\alpha|}\to \theta}\tau_{K,\alpha}\geq  \limsup_{\theta'\to \theta, \ \theta'\in \Sigma^0}\tau(K,\theta'). \end{equation}
To see this, we first choose a sequence $\{\theta_k\}\subset \Sigma^0$ with $\theta_k \to \theta$, and we then find a sequence of integers $\{i_k\}$ with $i_k\to \infty$ such that
$$|\frac{\alpha(i_k)}{|\alpha(i_k)|}-\theta_k|<1/k \ \hbox{and} \ |\tau_{K,\alpha(i_k)}-\tau(K,\theta_k)|< 1/k.$$
Then
$$\limsup_{\frac{\alpha}{|\alpha|}\to \theta}\tau_{K,\alpha}\geq \limsup_{k\to \infty} \tau_{K,\alpha(i_k)}=\limsup_{k\to \infty}\tau(K,\theta_k).$$
Since this holds for any sequence $\{\theta_k\}\subset \Sigma^0$ with $\theta_k \to \theta$, (\ref{limsup}) follows.

For easier reading, we will write our Chebyshev constants as $\tau_{\alpha}(K)$ since the multiindices $\alpha$ will be important. First, since $z_1 \prec ... \prec z_d$ and $K$ is circled, to compute $\tau_{(n,0,...,0)}(K)$, we minimize sup norms on $K$ over the {\it single} homogeneous polynomial $z_1^n$: thus $\tau_{(n,0,...,0)}(K)=\max_{(z_1,z')\in K} |z_1|$ and 
$$\lim_{n\to \infty}\tau_{(n,0,...,0)}(K)=\max_{(z_1,z')\in K} |z_1|$$
which is positive since $K$ is not pluripolar. Take a sequence $\alpha_j$ with $n_j:=|\alpha_j|\to \infty$ and $\theta_j:=\frac{\alpha_j}{|\alpha_j|}\to (1,0,...,0)$. We can write
                  $$\theta_j=(1-\epsilon_j,a_j) $$
                  where $\epsilon_j= |a_j| \to 0$. We proceed as follows. Writing $z=(z_1,z')=(z_1,z_2,...,z_d)\in \C^d$, competitors for $\tau_{(0,n_j a_j)}(K)$ are homogeneous polynomials of the form
$$p_{0,n_j a_j}(z_1,z'):=z'^{n_ja_j}+...+az_1^{n_j\epsilon_j}.$$
Consider $\tau_{(n_j(1-\epsilon_j),n_ja_j)}(K)$ where we minimize sup norms on $K$ over the homogeneous polynomials of the form
$$z_1^{n_j(1-\epsilon_j)}z'^{n_ja_j}+...+ az_1^{n_j}.$$ 
Some of these can be factored as
                  $$z_1^{n_j(1-\epsilon_j)}z'^{n_ja_j}+...+ az_1^{n_j}=z_1^{n_j(1-\epsilon_j)}(z'^{n_j a_j}+...+az_1^{n_j\epsilon_j});$$
                   i.e., the competitors for $\tau_{(n_j(1-a_j),n_j\epsilon_j)}(K)$ include $z_1^{n_j(1-\epsilon_j)}$ times the competitors for $\tau_{(0,n_j a_j)}(K)$. Hence
                    $$\tau_{(n_j(1-\epsilon_j),n_j a_j)}(K)^{n_j}\leq \tau_{(0,n_j a_j)}(K)^{n_j\epsilon_j}\cdot [\max_{(z_1,z')\in K} |z_1|^{n_j(1-\epsilon_j)}]$$
                  showing that 
                  $$\limsup_{j\to \infty}\tau_{(n_j(1-\epsilon_j),n_j a_j)}(K)\leq \limsup_{j\to \infty}\bigl( \max_{(z_1,z')\in K} |z_1|^{1-\epsilon_j}\cdot \tau_{(0,n_j a_j)}(K)^{\epsilon_j}\bigr)=\max_{(z_1,z')\in K} |z_1|.$$
                  Thus we have shown that 
                  \begin{equation}\label{Rlim} \limsup_{\frac{\alpha}{|\alpha|}\to (1,0,...,0)}\tau_{K,\alpha}\leq \max_{(z_1,z')\in K} |z_1|=\lim_{n\to \infty}\tau_{(n,0,...,0)}(K).\end{equation}

We will show that  
\begin{equation} \label{lastep} \liminf_{\theta \to (1,0,...,0), \ \theta \in \Sigma^0} \tau(K,\theta)\geq \lim_{n\to \infty}\tau_{(n,0,...,0)}(K).\end{equation}
Given this, from (\ref{liminf}), $\liminf_{\frac{\alpha}{|\alpha|}\to (1,0,...,0)}\tau_{K,\alpha}= \liminf_{\theta \to (1,0,...,0), \ \theta \in \Sigma^0} \tau(K,\theta)$ so that 
$$\liminf_{\frac{\alpha}{|\alpha|}\to (1,0,...,0)}\tau_{K,\alpha}\geq  \lim_{n\to \infty}\tau_{(n,0,...,0)}(K).$$
On the other hand, from (\ref{limsup}), $\limsup_{\frac{\alpha}{|\alpha|}\to (1,0,...,0)}\tau_{K,\alpha}\geq  \limsup_{\theta \to (1,0,...,0), \ \theta \in \Sigma^0}\tau(K,\theta)$ and (\ref{Rlim}) gave us 
$$\lim_{n\to \infty}\tau_{(n,0,...,0)}(K)\geq \limsup_{\frac{\alpha}{|\alpha|}\to (1,0,...,0)}\tau_{K,\alpha}.$$
Putting these last two displayed inequalities together gives us that
$$\lim_{\frac{\alpha}{|\alpha|}\to (1,0,...,0)}\tau_{K,\alpha}$$
exists and equals $\lim_{n\to \infty}\tau_{(n,0,...,0)}(K)$.

It remains to prove (\ref{lastep}). We may assume $\hat K =K=\{z\in \C^d: \bar \rho_K \leq 0\}$ and the intersection of $K$ with any complex line through the origin is a closed disk centered at the origin in this complex line of positive radius. Suppose first that for $K\cap \{z'=0\}=\{(z_1,0)\in K\}$ we have
$$\max\{|z_1|: (z_1,0) \in K\}=\max_{(z_1,z)\in K}|z_1|=:R.$$
Recall $\lim_{n\to \infty}\tau_{(n,0,...,0)}(K)=\max_{(z_1,z')\in K}|z_1|=R$. We can use interior approximation by an ellipsoid to get the lower bound for the $(1,0,...,0)$ Chebshev constant of $K$. To be precise, by our hypothesis, for $r<R$, the closed disk 
$$\{(z_1,0):|z_1|\leq r\}$$
is contained in the interior of $K$ and hence for some constant $A=A(r)$ the ellipsoid
$$E=E(A,r):=\{(z_1,z'): |z_1|^2/r^2 + |z'|^2/A^2 \leq 1\}$$
is contained in the interior of $K$. Thus for all $\theta \in \Sigma^0$ we have 
$$\tau(K,\theta) \geq \tau(E,\theta).$$
Explicit calculation gives that
$$\lim_{\theta \to (1,0,...,0), \ \theta \in \Sigma^0} \tau(E,\theta)= r.$$
Hence
$$\liminf_{\theta \to (1,0,...,0), \ \theta \in \Sigma^0} \tau(K,\theta)\geq r$$
for all $r<R$ so that 
$$\liminf_{\theta \to (1,0,...,0), \ \theta \in \Sigma^0} \tau(K,\theta)\geq R=\lim_{n\to \infty}\tau_{(n,0,...,0)}(K),$$
establishing (\ref{lastep}) in this case.

Finally, suppose that 
$$\max\{|z_1|: (z_1,0) \in K\} < \max_{(z_1,z')\in K}|z_1|.$$
Let the right-hand-side be attained at the point $(c,d')\in K$. Consider the change of coordinates
$(\tilde z_1,\tilde z'):=(z_1,z'- z_1d'/c)$ and let $\tilde K$ be the image of $K$. The Chebyshev constants of $K$ and $\tilde K$ are the same as this transformation preserves the classes of polynomials used to define
the directional Chebyshev constants (recall $z_1$ was first in the order $\prec$). But now for $\tilde K\cap \{\tilde z'=0\}=\{(\tilde z_1,0)\in \tilde K\}$ we have
$$|c|=\max\{|\tilde z_1|: (\tilde z_1,0) \in \tilde K\}=\max_{(\tilde z_1,\tilde z')\in \tilde K}|\tilde z_1|$$
and we are back in the first case. Thus we have established (\ref{lastep}) in all cases: 
$$\liminf_{\theta \to (1,0,...,0), \ \theta \in \Sigma^0} \tau(K,\theta)\geq R=\lim_{n\to \infty}\tau_{(n,0,...,0)}(K).$$

\end{proof}

We summarize the steps of the proof of Theorem \ref{maint}: 

\begin{proof} Recall that for $K\subset \C$ compact (not necessarily regular) and $w$ admissible on $K$, the limit $\lim_{n\to \infty} \tau_{K,n}^w$ of weighted Chebyshev constants 
$$\tau_{K,n}^w:=\bigl(\inf \{||w^np_n||_K: p_n(z) = z^n + ...\}\bigr)^{1/n}$$
exists (cf., Theorem III.3.1 in \cite{SaTo}). This begins the induction; Theorem \ref{maint} holds if $d=1$ in both the weighted and unweighted ($w\equiv 1$) setting. We remark that in $\C$, 
the notions of regularity and local regularity coincide.

From Proposition \ref{step3} and Lemma \ref{100}, the limit (\ref{wbdrydirl}) in the weighted case existing for $K$ compact and locally regular in $\C^{d-1}$ and $w$ admissible and  and continuous on $K$ implies the limit in (\ref{bdrydirl}) exists for $K$ circled and regular in $\C^d$. Then, since $K\subset \C^d$ regular implies the circled set $K_{\rho}$ is regular in $\C^d$, from Proposition \ref{step1}, the limit in (\ref{bdrydirl}) existing for $K$ circled and regular in $\C^d$ implies the limit in (\ref{bdrydirl}) exists for $K$ compact and regular in $\C^d$. Finally, since $K\in\mathbb{C}^d$ locally regular and $w=e^{-Q}$ continuous implies $V_Z$ is continuous where recall
$$Z:=\{z\in\mathbb{C}^d : V_{K,Q}(z) \leq ||V_{K,Q}||_K\},$$
 from Proposition \ref{step2}, the limit in (\ref{bdrydirl}) existing for $K$ compact and regular in $\C^d$ implies the limit (\ref{wbdrydirl}) in the weighted case exists for $K$ compact and locally regular in $\C^d$ and $w$ admissible and continuous on $K$. This completes the inductive step.

\end{proof}

\section{Appendix}  We state and proof Lemmas \ref{L} and \ref{rafal} used in the proof of Proposition \ref{step3}.

\begin{lemma} \label{L} Let $K\subset \C^d$ be a regular, circled compact set with $\hat K =K=\{z\in \C^d: \bar \rho_K \leq 0\}$. Fix $\epsilon >0$ and let $K_{\epsilon}:=\{z\in \C^d: \bar \rho_K(z) \leq-\epsilon\}$. Then for some $\eta =\eta(\epsilon) >0$, the polynomial convex hull 
$\widehat{S}$ of $S:=S(\epsilon,\eta)=K\setminus \{|z_1|<\eta\}$ contains $K_{\epsilon}$.

\end{lemma} 

\begin{proof} Since $\bar \rho_K$ is continuous and $K$ is the closure of 
the balanced domain of holomorphy $\{z\in \C^d: \bar \rho_K < 0\}$, it follows (cf., pp. 187-188 of \cite{K}) that for all $\epsilon >0$, $K_{\epsilon}$ equals its (homogeneous) polynomial 
hull. To prove the result, it suffices to show that there exists $\eta = \eta (\epsilon)>0$ such that for all $z=(z_1,z')\in \partial K_{\epsilon}$, we have $z\in \widehat {S}$.

 By uniform continuity of $\bar \rho_K$ on compact sets, given $\epsilon >0$, there exists $\delta >0$ such that for all $z\in  \partial K_{\epsilon}$, if $\tilde z \in\C^d$ satisfies $|z-\tilde z| < \delta$, then $z\in K$ (i.e., if $\bar \rho_K(z) = -\epsilon$, then $\bar \rho_K(\tilde z) <0$).  
 Fix $\eta < \delta/4$ and define $S:=S(\epsilon,\eta)$ as above. Take $z=(z_1,z')=(z_1,z_2, z_3,...,z_d) \in  \partial K_{\epsilon}$. If $|z_1| \geq \delta/4$ then $z\in S$. Thus suppose $|z_1|< \delta/4$. For each $\tilde z =(\tilde z_1,\tilde z')$ with $|z-\tilde z| < \delta$ we have $\tilde z \in K$. In particular, for $\psi\in [0,2\pi]$, consider all points of the form 
 $$\{z_{\psi}(t):= (\frac{\delta}{4} e^{i\psi}, z_2 + t \delta/4, z_3,...,z_d): t\in \C, \ |t|=1\}.$$
 This is a circle of radius $\delta/4$ centered at the point $(\frac{\delta}{4}e^{i\psi}, z_2, z_3,...,z_d)=(\frac{\delta}{4} e^{i\psi}, z')$ lying in the complex plane $\{(\xi_1,...,\xi_d):\xi_1=\frac{\delta}{4}e^{i\psi}, \ \xi_3=z_3, \ ..., \ \xi_d=z_d\}$. Note that $|z_{\psi}(t) -z| < \delta$ and hence $z_{\psi}(t) \in K$ for all $t\in \C$ with  $|t|=1$. Moreover, since $\delta/4 > \eta$, we have $z_{\psi}(t) \in S$. This shows that the center $(\frac{\delta}{4}e^{i\psi}, z_2, z_3,...,z_d)=(\frac{\delta}{4} e^{i\psi}, z')$ belongs to the polynomial hull $ \widehat {S}$ (more generally, for a compact set $K$, if there is an analytic disk $f: \Delta=\{t\in \C: |t|<1\} \to\C^d$ with $f(0)=z_0$ and $f(\partial \Delta)\subset K$, then $z_0\in \hat K$). But this holds for each $ \psi\in [0,2\pi]$. Thus the whole circle 
 $$\{z_{\psi}:= (\frac{\delta}{4} e^{i\psi}, z'): \psi\in [0,2\pi]\}$$
 belongs to $ \widehat {S}$. But this circle bounds the disk
 $$\{(w_1, z'):|w_1|\leq \delta/4\}$$
 which thus is also contained in $ \widehat {S}$. In particular, $z=(z_1,z')\in \widehat {S}$.

 \end{proof}
 
 The second result is due to Rafal Pierzchala. Theorem 5.3.9 of \cite{K} shows that the notion of regularity for a compact set is preserved under nondegenerate holomorphic mappings. Here, $f$ is nondegenerate on an open set $U$ if the complex Jacobian of $f$ has maximal rank on a dense set in $U$. Pierzchala's result shows that the same is true for local regularity. Recall that local regularity of a compact set at a point implies regularity but the converse is not necessarily true.
 
 \begin{lemma} \label{rafal} Let $f:U\to \C^m$ be a holomorphic mapping on an open set $U\subset \C^d$ and let $K\subset \C^d$ be compact with $\hat K \subset U$. Let $a\in K$ and let $U(a)$ be the connected component of $U$ containing $a$. If $K$ is locally regular at $a$ and $f$ is nondegenerate in a neighborhood of $a$, then $f(K)$ is locally regular at $f(a)$.
 
 \end{lemma}
 
 \begin{proof} We want to show that $\lim_{w\to f(a)}V_{f(K)\cap B(f(a),r)}(w)=0$. Fix $r_0$ with $B(a,r_0)\subset U$ and take $M>0$ with 
$$ |f(\zeta)-f(a)|\leq M|\zeta -a| \ \hbox{for all} \ \zeta \in B(a,r_0). 
$$
 First suppose $r<Mr_0$ and let $E:=K\cap B(a,r/M)$. Note that $\hat E \subset \hat K\cap B(a,r/M)$. From Theorem 5.3.9 of \cite{K} we have 
  \begin{equation} \label{reginv} \lim_{w\to f(a)}V_{f(E)}(w)=0. 
   \end{equation}
 Since $f(E)\subset f(K\cap B(a,r/M))$ we have
   \begin{equation} \label{easy} V_{f(K)\cap B(f(a),r)}\leq V_{f(E)}. 
   \end{equation}
   Thus, from (\ref{reginv}) and (\ref{easy}),
   $$0 \leq \lim_{w\to f(a)}V_{f(K)\cap B(f(a),r)}(w)\leq \lim_{w\to f(a)} V_{f(E)}(w)=0,$$
   yielding the result when $r<Mr_0$.
   
   If $r\geq Mr_0$, we observe that 
   $$f(K)\cap B(f(a),r_0M/2) \subset f(K)\cap B(f(a),r);$$
   hence applying the previous case with $r_0M/2$ in place of $r$ we obtain 
   $$0 \leq \lim_{w\to f(a)}V_{f(K)\cap B(f(a),r)}(w)\leq \lim_{w\to f(a)} V_{f(K)\cap B(f(a),r_0M/2)}(w)=0.$$
   This shows that $\lim_{w\to f(a)}V_{f(K)\cap B(f(a),r)}(w)=0$.

 \end{proof}

{\bf Authors:}\\[\baselineskip]
T. Bloom, bloom@math.toronto.edu\\
University of Toronto, Toronto, CA\\
\\[\baselineskip]
N. Levenberg, nlevenbe@iu.edu\\
Indiana University, Bloomington, IN, USA\\

\end{document}